\documentclass[12pt, 14paper,reqno]{amsart}
\usepackage{amsmath,amsfonts,amssymb}

\theoremstyle{plain}
\newtheorem{theorem}{Theorem}
\newtheorem{lemma}{Lemma}

\theoremstyle{definition}

\theoremstyle{remark}
\newtheorem{remark}{Remark}

\numberwithin{equation}{section}
\numberwithin{lemma}{section}
\numberwithin{theorem}{section}
\usepackage{amsmath}
\usepackage{amsfonts}   
\usepackage{amssymb}
\usepackage{amssymb, amsmath, amsthm}
\usepackage[breaklinks]{hyperref}

\usepackage{graphicx}
\usepackage{amsthm}

\begin{document} \title[A family of elliptic curves]{Integral solutions of certain Diophantine equation in quadratic fields
}
\author{Richa Sharma}

\address{Richa Sharma @Kerala School of Mathematics, Kozhikode-673571, Kerala, India}
\email{richa@ksom.res.in}

\keywords{ Elliptic curves, Diophantine equation}
\subjclass[2010] {11D25, 11D41, 11G05}
\maketitle

\begin{abstract} 

\noindent 
Let $K= \mathbf{Q}(\sqrt{d})$ be a quadratic field and $\mathcal{O}_{K}$ be its ring of integers. We study the solvability of the Diophantine equation $r + s + t = rst = 2$ in $\mathcal{O}_{K}$. We prove that except for $d= -7, -1, 17$ and $101$
this system is not solvable in the ring of integers of other quadratic fields.
\end{abstract}

\section{\textbf{Introduction}}
In 1960, Cassels \cite{Cassels} proved that the system of equations 
\begin{equation} \label{a}
r + s + t = r s t = 1,
\end{equation}
is not solvable in rationals $r,s$ and $t$. 
Later in 1982, Small \cite{Charles} studied the solutions of \eqref{a} in the rings $\mathbb{Z}/m\mathbb{Z}$ and in the finite fields $F_{q}$ where $q = p^{n}$ with $p$ a prime and $n \ge 1$. Further in 1987, Mollin et al. \cite{Mollin} 
considered \eqref{a} in the ring of integers of $K=\mathbf{Q}(\sqrt{d})$ 
and  proved that solutions exist if and only if $d=-1,2$ or $5$, where $x,y$ and $z$ are units in $\mathcal{O}_K$.
 Bremner \cite{Cubic, Quartic} in a series of two papers determined all cubic and quartic fields whose ring of integers contain a solution to \eqref{a}.
Later in 1999, Chakraborty et al. \cite{Kalyan} also studied \eqref{a} in the ring of integers of quadratic fields reproducing the findings of Mollin et al. \cite{Mollin} for the original system by adopting a different technique.

Extending the study further, we consider the equation
\begin{equation}  \label{1}
r + s + t = rst = 2.
\end{equation}
The sum and product of numbers equals $1$ has natural interest where as sum and product equals other naturals is a curious question. The method adopted here may not be suitable to consider a general $n$ instead of $2$ as for each particular $n$ the system give rise to a particular elliptic curve which may have different `torsion' and `rank' respectively. The next case, i.e. when the sum and product equals to $3$ is discussed in the last section.

To begin with we perform suitable change of variables and transform \eqref{1} to an  elliptic curve with the Weierstrass form
\begin{equation} \label{2}
 E_{297}: Y^2=X^3+135 X+297
\end{equation}
and then study $E_{297}$ 
in the ring of integers of $K = \mathbb{Q}(\sqrt{d})$.
\begin{remark}
We transform \eqref{1} into an elliptic curve \eqref{2} to show that one of the $(r,s,t)$ has to belong to $\mathbb{Q}$ (shown in \S3). 
\end{remark}

System \eqref{1} give rise to the quadratic equation
$$
x^{2}-(2-r)x+\frac{2}{r}=0,~r \neq 0,
$$
with discriminant 
\begin{equation} \label{r}
    \Delta = \frac{r(r^3-4r^2+4r-8)}{r}.
\end{equation}
At hindsight there are infinitely many choices for the quadratic fields contributed by each $r$ of the above form where the system could have solutions.
The main result of this article is that the only possibilities are $r = \pm 1, 2$ and $-8$. Thus \eqref{1} is solvable only in $K=\mathbf{Q}(\sqrt{d})$ with $d = -7, -1, 17$ and $101$.
Also the solutions are explicitly given. Throughout this article we denote ‘the point at infinity' of an elliptic curve by ${\mathcal{O}}$.
 Now we state the 
 main result of the paper. 
\begin{theorem} \label{thm1}
 Let $ K = \mathbb{Q}(\sqrt{d})$ be a quadratic field 
and $\mathcal{O}_{K}$ denote its ring of integers.  
Then the system
$$
r + s + t = rst = 2
$$ 
has no solution in $\mathcal{O}_K$ except for $d  = -7, -1, 17$ and $ 101$.
\end{theorem}
In \S 4 we discuss the rank of $E_{297}$ in $\mathbb{Q}$ and in quadratic fields of our interest.
\section{Preliminaries}
In this section we mention some results which are needed for the proof of the Theorem \ref{thm1}. First we state a basic result from algebraic number 
theory.\\
\begin{theorem}\label{rs1}
Let $K=\mathbb{Q}(\sqrt{d})$ with $d$ a square-free integer, then 
$$
\mathcal{O}_K=\begin{cases}
\mathbb{Z}[\frac{1+\sqrt{d}}{2}] {\ \text{ if }\ d\equiv 1\pmod 4,}\\
\mathbb{Z}[\sqrt{d}]~~ {\ \text{ if }\ d\equiv 2, 3\pmod 4.}
\end{cases}
$$ 
\end{theorem}
We study the solutions of a family of elliptic curves defined over $\mathbb{Q}$ in the ring of integers of a quadratic field.

Let $K= \mathbb{Q}(\sqrt{d})$ and
$s'$ be the conjugate of an element $s\in K$ over $\mathbb{Q}$.
Further $R = \mathcal{O}_{K}[S^{-1}]$, where $S$ is some finite set of primes in $\mathcal{O}_{K}$. 
Thus $ \mathcal{O}_{K} \subset \mathcal{O}_{K}[S^{-1}] \subset K$. \\
Laska \cite{Laska} considered the equation (for $r \in \mathbb{Z}$ and $r \neq 0$)
$$
\Gamma_{r}:~ y^2 = x^3 -r,
$$
which defines the Weierstrass form of an elliptic curve (call it $E_{r}$) over $\mathbb{Q}$. 
For an elliptic curve $E$ over $\mathbb{Q}$ the ``trace map" 
$$
\sigma:E(K) \longrightarrow E(\mathbb{Q})
$$ 
is given by 
$$
\sigma(\mathcal{P}) = \mathcal{P} \oplus \mathcal{P}'.
$$
Here $\mathcal{P}'$ is the conjugate of the element $\mathcal{P} \in E(K)$ arising from the conjugation in $K$ and $\oplus$ is the usual elliptic curve addition.
Laska considered the ``trace map'' for $\Gamma_r$ and calls it
$$
\sigma_{r,R}: \Gamma_{r}(R) \rightarrow  \Gamma_{r}(\mathbb{Q}) \cup \{\mathcal{O}\}.
$$
If $r,R$ are fixed we simply write $\sigma$ instead of $\sigma_{r,R}$. The aim was to study $\sigma^{-1}(P)$ for a given $P \in \Gamma_{r}(\mathbb{Q}) \cup \{\mathcal{O}\}$. He  divided this inverse image set into two parts and called them `exceptional' and `non-exceptional' respectively. These two sets consist of:\\
$\bullet$~ an element $\mathcal{P} = (s,t) \in \Gamma_{r}(R)$ with $s \neq s'$ is called an exceptional solution of $\Gamma_{r}$ in $R$,\\
$\bullet$~non-exceptional solutions of $\Gamma_{r}$ in $R$ are contained in $\sigma^{-1}(P)$ for a given $P \in \Gamma_{r}(\mathbb{Q}) \cup \{\mathcal{O}\}$.

If $\mathcal{P}=(s,t) \in \Gamma_{r}(R)$ is a non-exceptional solution
then $s=s'$ and from the Weierstrass equation that implies  $t= \pm t'$. Thus if $P = \mathcal{O}$, then
all candidates in $\sigma^{-1}(P)$ are non-exceptional. More precisely, one can show that, 
\begin{eqnarray}
&&
\sigma^{-1}(\mathcal{O})= \{(z^{-1}p, z^{-2}q\theta) : (p, q) \in \Gamma_{z^{3}r}(R \cap \mathbb{Q}),  \nonumber\\ 
&&\hspace*{20mm} p \in z(R \cap \mathbb{Q}), q \in z^2(R \cap \mathbb{Q})\}
\end{eqnarray}
where for simplicity it is assumed that $\theta^{-1} \in R$.

If $P \neq \mathcal{O}$, then the non-exceptional solutions $\mathcal{P}$ contained in $\sigma^{-1}(P)$ are
exactly given by the conditions 
$$
\mathcal{P} \in \Gamma_{r}(R \cap \mathbb{Q}), \mathcal{P} \oplus \mathcal{P} = P.
$$
Thus the non-exceptional
solutions of $\Gamma_{r}$ in $R$ which are contained in $\sigma^{-1}(P)$ are obtained by solving the
equation $\Gamma_{z^{3}r}$ respectively $\Gamma_{r}$ in $R \cap \mathbb{Q}$.
Hence either $\mathcal{P} \in \Gamma_{r}(\mathbb{Q})$  or $\mathcal{P} \oplus \Bar{\mathcal{P}} = \mathcal{O}$.

Here we substitute $\Gamma_{r}$ by $E_{297}$ and $R$ by $\mathcal{O}_{K}$ to study the solutions of \eqref{1} in $\mathcal{O}_{K}$ by pulling back the elements of $E_{297}(\mathbb{Q})$ using the above mentioned technique.
\section{Proof of Theorem \ref{thm1}}
\begin{proof} Let us first transform the system \eqref{1} to the desired Weierstrass form $E_{297}$.
The system is
\begin{equation*}  \label{b}
r + s + t = rst = 2.
\end{equation*}
Putting the value of $t$ upon simplification it becomes
\begin{equation*} 
s + r + \frac{2}{rs} = 2.
\end{equation*}
Now substituting $r = - 2/x$ and $ s = -y/x $ in the last equation give
\begin{equation} \label{c}
    y^{2} + 2y + 2xy  = x^{3}.
\end{equation}
 Putting $x = x_{1}$ - $  1/2$ and $ y = \frac{y_{1}}{2} - x_{1} - \frac{1}{2}$ in \eqref{c} rids it from the $xy$ term 
 \begin{equation} \label{d}
 y_{1}^{2} = 4x_{1}^{3} - 2 x_{1}^{2} + 7 x_{1} + \frac{1}{2}.
\end{equation}
Further substituting  $ x_{1}= x_{2} + \frac{1}{6} $ and $y_{1}=y_{2}$ rids \eqref{d} the $x_{1}^{2}$ term
\begin{equation} \label{e}
y_{2}^{2}= 4x_{2}^{3} + \frac{20}{3} x_{2} + \frac{44}{27}.
\end{equation}
Now again putting $y_{2}= Y_{1}/27$ and $x_{2} = X_{1}/9$ give
\begin{equation} \label{f}
Y_{1}^{2} = 4X_{1}^{3} + 540 X_{1} + 1188.
\end{equation}
Finally substituting $Y_{1}=2Y$ and $X_{1}=X$ get us the required Weierstrass form
\begin{equation} \label{g}
E_{297}: Y^2=X^3+135 X+297.
\end{equation}
Here with the help of SAGE \cite{Ss45} we can conclude that 
$$
E_{297}(\mathbb{Q})_{tors} \cong \mathbb{Z}_{3} = \left\lbrace {\mathcal{O},(3,\pm{27}) }\right\rbrace
$$ 
and the $\mathbb{Q}$-rank of $E_{297}$ is zero (we give a mathematical proof of this fact in \S4).

Thus $\mathcal{O}$ and $(3,\pm{27})$ are the only $\mathbb{Q}$  rational points of \eqref{g}.
It is not difficult to see that the inverse transformation
$$
r = 18/(3-X)  ~~\mbox{and}~~  s = (Y - 3X - 18)/(3(3 - X))
$$
allows us to pass from \eqref{g} to \eqref{1}. Before proceeding to final leg of the proof of Theorem \ref{thm1}, we make a couple of claims and give their proofs.

\noindent {\textbf{Claim 1}}: One of the $(r,s,t)$ satisfying \eqref{1}
must belong to $\mathbb{Q}$.
\begin{proof}(Proof of claim 1)
Two cases needed to be considered.

Case I: $\mathcal{P}$ is non-exceptional. In this case either 
 $\mathcal{P} \in E_{297}(\mathbb{Q})$  or $\mathcal{P} \oplus \Bar{\mathcal{P}} = \mathcal{O}$.
If $\mathcal{P} \in E_{297}(\mathbb{Q})$ then this implies $r \in \mathbb{Q} $. 

Let $\mathcal{P} \oplus \Bar{\mathcal{P}} = \mathcal{O}$ and
$\mathcal{P} =(a + b \sqrt{d}, k+l \sqrt{d})$.
As $\mathcal{P}$ is non-exceptional, $b = 0$ and $k=0$. Thus $\mathcal{P}= (a, l \sqrt{d})$ and in this case too
$$
r = 18/(3-a) \in \mathbb{Q}.
$$
Case II:
Now if $\mathcal{P}$ is exceptional then $\mathcal{P} + \Bar{\mathcal{P}} = (3, \pm 27)$.
The curve \eqref{g} has exactly three elements over $\mathbb{Q}$ and the non-trivial elements are
of order $3$. Lets call them $\mathcal{O}, \Omega$ and $2 \Omega$. If $\mathcal{P} + \Bar{\mathcal{P}} = \Omega$, then clearly $\mathcal{P} + \Omega$ is
non-exceptional, since
$$
(\overline{\mathcal{P} + \Omega} )+ \mathcal{P} + \Omega = \overline{\mathcal{P}} + \mathcal{P} + 2\Omega = 3 \Omega = \mathcal{O}.
$$
Now if $\mathcal{P} + \overline{\mathcal{P}} = 2\Omega$, as before we can show that $\mathcal{P} + 2 \Omega$ is also non-exceptional.
As the claim is valid for non-exceptional elements, it is true
for $\mathcal{P} + \Omega$ and $\mathcal{P} + 2 \Omega$. Hence it is true for $\mathcal{P}$ itself.
\end{proof}
\noindent Hence without loss of generality we assume that $r \in  \mathbb{Q}$.

\noindent {\textbf{Claim 2}}
The only possibilities for $r$ satisfying the system 
$$
r + s + t = rst = 2
$$ 
in $\mathcal{O}_K$ are $\pm 1, 2$ and $-8$.
\begin{proof}(Proof of claim 2)
As $r \in  \mathbb{Q}$ and we are looking for solutions in $\mathcal{O}_{K}$, this would imply that $r \in \mathbb{Z}$.
Three possibilities needed to be considered:
\begin{itemize}
    \item If $r= \pm 1$.
    In this case solutions exist.
    \item If $r$ is odd. In this case the denominator of \eqref{r} will be multiple of an odd number but in $\mathcal{O}_{K}$ denominator is only $2$ or $1$ by Theorem \ref{rs1}.
    So in this case there do not exist any solution in $\mathcal{O}_{K}$.
    \item If $r$ is even. In this case save for $r=2$ and $-8$, 
    the denominator of \eqref{r} will be multiple of $2$ and in some other cases denominator is $1$ but $d \equiv 1 \mod 4$. Thus again by Theorem \ref{rs1}, in this case too we don't (except for $r= 2, -8$) get any solution in $\mathcal{O}_{K}$.
\end{itemize}
Thus except these values of $r = \pm 1, 2$ and $-8$, this system of equation is not solvable in the ring of integers of other quadratic fields.

\end{proof}
\noindent We are now in a position to complete the proof of Theorem \ref{thm1}.
We deal with all the four possibilities of $r$ separately. 

\noindent When $r = 1$, from \eqref{b} we have $s+t = 1$ and $st = 2$. Thus we get 
$$
(s, t) = (\frac{1-\sqrt{-7}}{2},\frac{1+\sqrt{-7}}{2})
 ~\mbox{and}~ (\frac{1+\sqrt{-7}}{2},\frac{1-\sqrt{-7}}{2}).
$$
When $r = -1$, we have $s + t = 3$ and $st = -2$. In this case
$$
(s, t) = (\frac{3-\sqrt{17}}{2},\frac{3+\sqrt{17}}{2}) 
~\mbox{and}~ (\frac{3+\sqrt{17}}{2},\frac{3-\sqrt{17}}{2}).
$$
Similarly when $r=2$ and $-8$, we get
$$
(s,t)=
(i,-i) (-i,i), (\frac{10+\sqrt{101}}{2}, \frac{10-\sqrt{101}}{2})
$$
$$
\mbox{and}~
(\frac{10-\sqrt{101}}{2}, \frac{10+\sqrt{101}}{2}).
$$
To conclude
$$
(1,\frac{1-\sqrt{-7}}{2},\frac{1+\sqrt{-7}}{2}), (1,\frac{1+\sqrt{-7}}{2},\frac{1-\sqrt{-7}}{2}),
$$
$$
(-1,\frac{3-\sqrt{17}}{2},\frac{3+\sqrt{17}}{2}), (-1, \frac{3+\sqrt{17}}{2},\frac{3-\sqrt{17}}{2})
$$ 
$$
(-8,\frac{10+\sqrt{101}}{2}, \frac{10-\sqrt{101}}{2})
(-8,\frac{10-\sqrt{101}}{2}, \frac{10+\sqrt{101}}{2})
$$
$$
(2,i,-i)~\mbox{and}~ (2,-i,i). 
$$
are the only solutions of \eqref{b} in $\mathcal{O}_K$.
\end{proof}
\section{Rank of $E_{297}$}
In this section we discuss the rank of $E_{297}$ over $\mathbb{Q}$ and over $\mathbb{Q}(\sqrt{d})$ for $d=-7, -1, 17$ and $101$.
\begin{lemma}The rank of $E_{297}(\mathbb{Q})$ is zero.
\end{lemma}
\begin{proof}
\noindent 
If possible let  rank $E_{297}(\mathbb{Q}) \neq 0$. This would imply that
\eqref{g} have rational solutions. Let
 $(X,Y) = (\frac{x_{1}}{x_{2}}, \frac{y_{1}}{y_{2}})$ with $(x_{1},x_{2})=(y_{1},y_{2})=1$, be one such solution. Now putting the values of $X$ and $Y$ in \eqref{g}, we obtain
\begin{equation} \label{h}
y_{1}^{2}x_{2}^{3} = x_{1}^{3} y_{2}^{2} + 135 x_{1} y_{2}^{2} x_{2}^{2} + 297 y_{2}^{2}x_{2}^{3}.
\end{equation}
Let $p$ be a prime such that $y_{2} = p^{\alpha} y_{22}$ where $\alpha \in \mathbb{Z}, \alpha \geq 1$ and $(p,y_{22}) =1 $. Putting the value of $y_{2}$ in \eqref{h}, gives that
right hand side of \eqref{h} is divisible by $p$. Therefore 
$
p \mid  y_{1}x_{2}. 
$
Which implies either $p \mid y_{1}$ or $p \mid x_{2}$.
Suppose $p \mid y_{1}$ then since $y_{2} = p^{\alpha} y_{22}$, we get a contradiction to the fact that $ (y_{1}, y_{2})$ = $1$.
Thus $p \mid x_{2}$ and we write $x_{2}= p^{\beta} x_{22}$ where $ \beta \in \mathbb{Z}, \beta \geq 1$ and $(p,x_{22}) =1 $. 
Now putting this $y_{2}$ and $x_{2}$ in \eqref{h}, gives 
\begin{equation}  \label{I}
    p^{3\beta} x_{22}^{3} y_{1}^{2} = p^{2\alpha} y_{22}^{2} x_{1}^{3} + 135  x_{1} p^{2\alpha + 2\beta} y_{22}^{2} x_{22}^{2} + 297 p^{2\alpha + 3\beta}  y_{22}^{2}x_{22}^{3}.
\end{equation}
Three cases can occur.\\
Case I. Suppose $3\beta > 2\alpha$. In this case, 
\begin{equation} \label{J}
    p^{3\beta - 2\alpha} x_{22}^{3} y_{1}^{2} =  y_{22}^{2} x_{1}^{3} + 135  x_{1} p^{2\beta} y_{22}^{2} x_{22}^{2} + 297 p^{3\beta}  y_{22}^{2}x_{22}^{3}.
\end{equation}
Further from \eqref{J}
$$
p \mid y_{22} x_{1}
$$
and forces $p \mid x_{1}$ as $(p, y_{22})= 1$, which is a contradiction to the fact that $(x_{1}, x_{2}) = 1$. \\
Case II. Suppose $3\beta > 2\alpha$. In that case $p \mid y_{1}$, which is also contradiction to the fact that $ (y_{1}, y_{2})  = 1$.
\\
Case III. This case is analogous to Case I.

Hence \eqref{J} has no solution in $\mathbb{Z}$ and that in turn implies \eqref{g} has no solution in $\mathbb{Q}$. Thus the rank of $E_{297}(\mathbb{Q})$ defined by the equation \eqref{g} is zero. 
\end{proof}
\begin{lemma}
The rank of $E_{297}(\mathbb{Q}(\sqrt{d}))$ for $d=-7, 17$is $1$ and for $d=101$ it is $2$.
\end{lemma}
\begin{proof}
 Let $E/K$ be an elliptic curve and $d \in  K^{*}$ be such that $L = K(\sqrt{d})$ is a quadratic
extension. Let $E_{d}/K$ be the twist of $E/K$, then by \cite{S92}, 
\begin{equation}\label{rs}
\text{rank}~E(L) = \text{rank}~E(K) + \text{rank}~ E_{d}(K).
\end{equation}
We conclude using SAGE \cite{Ss45} that 
(already we have noted that rank $E_{297}(\mathbb{Q}) = 0$) rank $E_{-7\cdot 297}(\mathbb{Q}) = 1$, rank $E_{-1\cdot 297}(\mathbb{Q}) = 1$, rank $E_{17\cdot 297}(\mathbb{Q}) = 1$ and rank $E_{101\cdot 297}(\mathbb{Q}) = 2$. 
Now using \eqref{rs} we have,
$$
\text{rank}~(E(\mathbb{Q}(\sqrt{-7})) =
\text{rank}~(E(\mathbb{Q}(\sqrt{-1})) =
\text{rank}~(E(\mathbb{Q}(\sqrt{17})) = 1.
$$
$$ \mbox{and}~
\text{rank}~(E(\mathbb{Q}(\sqrt{101})) = 2
$$
\end{proof}
\section{Concluding remarks} 
We showed that the system
$$
r + s + t = rst = 2
$$ 
has no solution in $\mathcal{O}_K$ except for $d  = -7,-1,17$ and $ 101$ and the solutions are explicitly given.

It would of interest to consider the next case, i.e,
\begin{equation} \label{q}
    r + s + t = rst = 3.
\end{equation}
Suitable change of variables transform \eqref{q} to an  elliptic curve with Weierstrass form
$$
E_{13122}: y^2 = x^3 + 3645 x - 13122.
$$
Torsion of $E_{13122}$ over $\mathbb{Q}$ is isomorphic to $\mathbb{Z}_{3}$ and it's rank is zero (using SAGE \cite{Ss45}).

We can conclude that except for $d = -2, -1,7,10$ and $13$, the system \eqref{q} has no solutions in the ring of integers of $K$ and the solutions can be explicitly given (following analogous argument). Analogously other individual cases can be treated.

\section{Acknowledgement} 
A part of this manuscript was completed while the author was visiting Prof. Sanoli Gun  at the Institute of Mathematical Sciences (IMSc), Chennai. She is grateful to Prof. Sanoli for her support and encouragement. It is a pleasure to Prof. Michel Waldschmidt for his valuable comments and suggestions during the visit to IMSc. 
Last but not the least, the serene ambience of Kerala School of Mathematics (KSoM) is a constant energy booster and is of great help to research.

\end{document}